\documentclass{amsart}

\usepackage{a4wide, amssymb, mathrsfs, stmaryrd}
\usepackage{fancyheadings}

\usepackage[
open,
openlevel=2,
atend
]{bookmark}[2009/08/13]

\usepackage{amsmath}

\usepackage[T1]{fontenc}

\usepackage[arrow,matrix,curve]{xy}

\usepackage{color}

\def\xto#1{\xrightarrow{#1}}

\newcommand{\nc}{\newcommand}

\newcommand{\myMO}[1]{{\fontshape{rm}{\textbf{#1}}}}

\DeclareMathOperator{\Grmod}{\myMO{Grmod}\hspace{+0.25ex}-\hspace{-0.25ex}}
\DeclareMathOperator{\Mod}{\myMO{Mod}\hspace{+0.25ex}-\hspace{-0.25ex}}

\DeclareMathOperator{\proj}{\myMO{proj.dim}}
\DeclareMathOperator{\gl}{\myMO{gl.dim}}

\nc{\Ch}{\operatorname{Ch}}
\nc{\sA}{{\mathscr A}}
\nc{\wt}{\widetilde} \nc{\bl}{\bullet} \nc{\al}{\alpha}
\nc{\sg}{\sigma} \nc{\vf}{\varphi} \nc{\om}{\omega}
\nc{\ve}{\varepsilon} \nc{\ol}{\overline} \nc{\lb}{\lambda}
\nc{\Lb}{\Lambda} \nc{\Gm}{\Gamma} \nc{\cP}{{\mathscr P}}
\nc{\sB}{{\mathscr B}}
\nc{\ul}{\underline} \nc{\os}{\overset} \nc{\us}{\underset}
\nc{\pa}{\partial} \nc{\wh}{\widehat} \nc{\sbs}{\subset} \nc{\br}{\breve}
\nc{\lra}{\longrightarrow} \nc{\all}{\allowdisplaybreaks}
\nc{\Ker}{\operatorname{Ker}} \nc{\Img}{\operatorname{Im}}
\nc{\Kan}{\operatorname{Kan}} \nc{\Hom}{\operatorname{Hom}}
\nc{\Imm}{\operatorname{Im}}   \nc{\Ho}{\operatorname{Ho}}
\nc{\Ext}{\operatorname{Ext}}    \nc{\Cone}{\operatorname{Cone}}
\nc{\pr}{\operatorname{pr}} \nc{\cls}{\operatorname{cls}}
\nc{\cof}{\operatorname{cof}}
\nc{\sSet}{\operatorname{\textbf{sSet}}}
\nc{\Map}{\operatorname{Map}}
\nc{\incl}{\operatorname{incl}}
\nc{\Hocolim}{\operatorname{Hocolim}}
\nc{\colim}{\operatorname{colim}}
\nc{\Endd}{\operatorname{End}}
\nc{\dimm}{\operatorname{dim}}
\nc{\const}{\operatorname{const}}
\nc{\inn}{\operatorname{in}}
\nc{\Ev}{\operatorname{Ev}}
\nc{\rr}{\operatorname{\textbf{r}}}
\nc{\cop}{\operatorname{\textbf{l}}}
\nc{\Sp}{\operatorname{Sp}}
\nc{\Comod}{\operatorname{\textbf{Comod}}}

\numberwithin{equation}{subsection}
\newtheorem{theo}[equation]{Theorem}
\newtheorem{lem}[equation]{Lemma}
\newtheorem{prop}[equation]{Proposition}
\newtheorem{coro}[equation]{Corollary}
\theoremstyle{definition}

\newtheorem{exmp}[equation]{Example}
\swapnumbers
\newtheorem{numbered paragraph}[equation]{}

\newcommand{\mybox}{\ensuremath \Box}

\begin{document}

\def\S{\textbf{S}}
\def\Z{{\mathbb Z}}
\def\L{{\mathcal L}}
\def\B{{\mathcal B}}
\def\Q{{\mathcal Q}}
\def\M{{\mathcal M}}
\def\D{{\mathcal D}}
\def\R{{\mathscr R}}
\def\E{{\mathcal E}}
\def\K{{\mathcal K}}
\def\W{{\mathcal W}}
\def\N{{\mathcal N}}
\def\T{{\mathcal T}}
\def\A{{\mathcal A}}
\def\C{{\mathcal C}}
\def\P{{\mathcal P}}
\def\V{{\mathcal V}}
\def\G{{\mathcal G}}
\let\xto\xrightarrow

\title[]
{On exotic equivalences and a theorem of Franke 
%and exotic triangulated equivalences 
}

\author{Irakli Patchkoria}

%\thanks{This research was supported by the Shota Rustaveli National Science Foundation grant  DI/27/5-103/12}

\begin{abstract} Using Franke's methods we construct new examples of exotic equivalences. We show that for any symmetric ring spectrum $R$ whose graded homotopy ring $\pi_*R$ is concentrated in dimensions divisible by a natural number $N \geq 5$ and has homological dimension at most three, the homotopy category of $R$-modules is equivalent to the derived category of $\pi_*R$. The Johnson-Wilson spectrum $E(3)$ and the truncated Brown-Peterson spectrum $BP\langle 2 \rangle$ for any prime $p \geq 5$ are our main examples. If additionally the homological dimension of $\pi_*R$ is equal to two, then the homotopy category of $R$-modules and the derived category of $\pi_*R$ are triangulated equivalent. Here the main examples are $E(2)$ and $BP \langle 1 \rangle$ at $p \geq 5$. The last part of the paper discusses a triangulated equivalence between the homotopy category of $E(1)$-local spectra at a prime $p \geq 5$ and the derived category of Franke's model. This is a theorem of Franke and we fill a gap in the proof.
  
\end{abstract}

\subjclass[2010]{55P42, 18G55, 18E30}
\keywords{module spectrum, ring spectrum, stable model category, triangulated category}
%derivator, }

\maketitle

\setcounter{section}{0}

\section{Introduction} 

\setcounter{subsection}{1}

Let $R$ be a symmetric ring spectrum and suppose that the homotopy ring $\pi_*R$ of $R$ is concentrated in dimensions divisible by some natural number $N \geq 2$. Further assume that the graded global homological dimension $\gl \pi_*R$ of $\pi_* R$ is less than $N$. Using Franke's methods from \cite{F96}, the paper \cite{Pat12} constructs a functor $\R \colon \D(\pi_* R) \to \Ho(\Mod R) $. Here $\Ho(\Mod R)$ is the \emph{homotopy category} of $\Mod R$, the model category of module spectra over $R$ and $\D(\pi_* R)$ is the \emph{derived category} of $\pi_* R$ which is by definition the homotopy category of differential graded $\pi_*R$-modules. The functor $\R$ has many interesting properties: It commutes with suspensions, is compatible with homology and homotopy groups, and restricts to an equivalence between the full subcategories of at most one-dimensional objects. In case when $\gl \pi_*R < N-1$, then $\R$ restricts to an equivalence between the full subcategories of at most two dimensional objects. As an 
application of the latter fact, one concludes that 
$\R \colon \D(\pi_* R) \to \Ho(\Mod R) $ is an equivalence of categories if $N \geq 4$ and $\gl \pi_*R \leq 2$. These facts are proved in \cite{Pat12}. In particular, the latter result improves a classification theorem by Wolbert \cite{W98}. 

This paper extends the results of \cite{Pat12} to homological dimension three. More precisely we prove the following:

\begin{theo} \label{main} Let $R$ be a symmetric ring spectrum such that $\pi_* R$ is concentrated in degrees divisible by a natural number $N \geq 2$ and assume that the graded global homological dimension of $\pi_*R$ is less than $N-1$. Then the functor 
 $\R \colon \D(\pi_* R) \to \Ho(\Mod R)$
restricts to an equivalence of the full subcategories of at most three dimensional objects.

\end{theo} 

As a consequence we obtain:

\begin{coro} \label{maincoro} Let $R$ be a symmetric ring spectrum such that $\pi_* R$ is concentrated in degrees divisible by a natural number $N \geq 5$ and assume that the graded global homological dimension of $\pi_*R$ is less than or equal to three. Then the functor 
 $\R \colon \D(\pi_* R) \to \Ho(\Mod R) $
is an equivalence of categories. 
\end{coro}

Examples to which this corollary applies are the Johnson-Wilson spectrum $E(3)$ and the truncated Brown-Peterson spectrum $BP\langle 2 \rangle$ at any prime $p \geq 5$.

The idea of the proof of Theorem \ref{main} is to combine Lemma \ref{presprojtr} below with the results of \cite{Pat12}. This is done using the octahedral axiom for triangulated categories.  

The proof of Theorem \ref{main} is very general and can be applied in many other contexts. In particular we get the following: Let $R$ be a symmetric ring spectrum such that $\pi_* R$ is concentrated in degrees divisible by a natural number $N \geq 5$ and assume that the graded global homological dimension of $\pi_*R$ is equal to two. Then there is an equivalence $$\R \colon \Ho((\Mod \pi_*R)^{\P}) \to \Ho((\Mod R)^{\P})$$ for any at most one dimensional poset $\P$. Here we just use that the homological dimension of the abelian category $(\Grmod \pi_*S)^{\P}$ is at most $3$. Using the results of \cite{Pat16}, one can see that these functors are compatible with the derivator structures. From this one obtains:

\begin{theo} \label{maintr} Let $R$ be a symmetric ring spectrum such that $\pi_* R$ is concentrated in degrees divisible by a natural number $N \geq 5$ and assume that the graded global homological dimension of $\pi_*R$ is equal to two. Then the equivalence 
 $\R \colon \D(\pi_* R) \to \Ho(\Mod R) $
is triangulated.
\end{theo}

This theorem improves some of the results of \cite{Pat16}. The main examples we have here are $BP\langle 1 \rangle$ and $E(2)$ for a prime $p \geq 5$. 

Methods (which are due to Franke \cite{F96}) used in this paper as well as in \cite{Pat12, Pat16} are very general. In particular, they also apply to $E(1)$-local spectra at a prime $p \geq 5$. The last section of the paper discusses a triangulated equivalence between the homotopy category of $E(1)$-local spectra at a prime $p \geq 5$ and the derived category of certain twisted chain complexes of $E(1)_*E(1)$-comodules. This fact is a well-known theorem by Franke \cite{F96} (see also \cite{R08}). The proof of it in \cite{F96} contains a gap. It is not clear from \cite{F96} that the constructed functor $\R$ preserves Adams resolutions which is necessary to compare Adams spectral sequences computing morphisms in the homotopy categories in question. Lemmas 6.2.1 and 6.2.4 of \cite{Pat12} and Lemma \ref{presprojtr} below translated to the context of comodules fill this gap for $E(1)$-local spectra for $p \geq 5$. However, in more general cases the gap still has to be filled. 

We do not provide proofs in the last section as they are very much identical to the proofs in \cite{Pat12, Pat16} and in Section \ref{mainsec}. For convenience of the reader we also prefer to give details in terms of module spectra and graded modules, rather than in terms of comodules and twisted chain complexes.  

\subsection*{Acknowledgements} I would like to thank Tobias Barthel, Jens Franke, Constanze Roitzheim, Tomer Schlank and Nathaniel Stapleton for useful conversations.   

This research was supported by the Danish National Research Foundation through the Centre for Symmetry and Deformation (DNRF92), the Shota Rustaveli Georgian National Science Foundation grant DI/27/5-103/12 and the German Research Foundation Schwerpunktprogramm 1786. 

\setcounter{subsection}{1}

\section{Preliminaries} 

\subsection{Graded homological algebra}

Let $B$ be a graded ring. We will denote by $\Grmod B$ the category of graded (right) $B$-modules. This category is an abelian category and has enough projective objects. For any graded $B$-module $M$, let $\proj M$ denote the (graded) \emph{projective dimension} of $M$ in $\Grmod B$. The \emph{graded global homological dimension} of $B$, denoted by $\gl B$, is the supremum of projective dimensions of objects in $\Grmod B$.

Next, let $N \geq 2$ be an integer. We say that a graded ring $B$ is $N$-\emph{sparse} if $B$ is concentrated in degrees divisible by $N$, i.e., $B_k=0$ if $k$ is not congruent to $0$ modulo $N$. 

If $A$ is a graded ring, then $\Mod A$ will denote the category of differential graded (right) modules over $A$.

\subsection{Stable model categories and triangulated categories}

We will freely use in this paper the language of triangulated categories (see e.g., \cite{GM96}) and model categories (see e.g., \cite{Q67, DS95, H99}). We will also assume that the reader is familiar with the notational conventions from \cite{Pat12} and \cite{Pat16}

Recall that a model category $\M$ is called a \emph{stable model category} if the \emph{suspension functor} $\Sigma \colon \Ho(\M) \to \Ho(\M)$ is an equivalence of categories. For any stable model category $\M$, the homotopy category $\Ho(\M)$ is triangulated with $\Sigma$ the suspension functor \cite[Proposition 7.1.6]{H99}. The distinguished triangles come from mapping cone sequences of the form
$$\xymatrix{X \ar[r]^f & Y \ar[r]^-{j} & \Cone(f) \ar[r]^\pa & \Sigma X,}$$
where $f$ is a morphism in $\M$ between cofibrant objects (cf. \cite[Section 2.4]{Pat12}).

Similarly, the triangulated structure on the derived category $\D(A)$ of a graded ring $A$ comes from the usual shift functor and algebraic mapping cone sequences
$$\xymatrix{ M \ar[r]^f & M' \ar[r]^\iota & C(f) \ar[r]^\partial & M[1],}$$
where $f$ is a chain map of differential graded $A$-modules (cf. \cite[Remark 2.4.8]{Pat12}).

For objects $X$ and $Y$ in a stable model category $\M$, the notation $[X,Y]$ will always stand for the abelian group of morphisms from $X$ to $Y$ in the homotopy category $\Ho(\M)$.

\subsection{Franke's functor}

In this subsection we review the construction of Franke's functor. We mainly follow \cite[Section 3.3]{Pat12} where the reader can also look up necessary details.

Let $\M$ be a simplicial stable model category and $S$ a compact generator of $\Ho(\M)$. Suppose that the graded ring $\pi_*S=[S,S]_*$ is $N$-sparse for $N \geq 2$ and the graded global homological dimension of $\pi_*S$ is less than $N$. The category $\Grmod \pi_*S$ splits as follows
$$\Grmod{\pi_*S} \sim \B \oplus \B[1] \oplus \ldots \oplus \B[N-1],$$ 
where $\B$ is the full subcategory of $\Grmod{\pi_*S}$ consisting of all those modules which are concentrated in degrees divisible by $N$. Recall that $\C_N$ denotes the poset
$$\xymatrix{\zeta_0
  \ar@{<-}[drrrr]%
    |<<<<<<<<<<<<{\text{\Large \textcolor{white}{$\blacksquare$}}}%
    |<<<<<<<<<<<<<<<<<<<<<{\text{\Large \textcolor{white}{$\blacksquare$}}}
    |>>>>>>>>>>>>>>>>>>>>>>>{\text{\Large \textcolor{white}{$\blacksquare$}}}
    |>>>>>>>>>>>>>>{\text{\Large \textcolor{white}{$\blacksquare$}}}
    |>>>>>>>>>>{\text{\Large \textcolor{white}{$\blacksquare$}}}%
& \zeta_1 & \ldots & \zeta_{N-2} & \zeta_{N-1} \\ \beta_0 \ar[u] \ar[ur] & \beta_1 \ar[u] \ar[ur] & \ldots \ar[ur] & \beta_{N-2} \ar[u] \ar[ur] & \beta_{N-1}. \ar[u] }$$
For a general diagram $X \in \M^{\C_N}$, we denote by
$$l_i \colon X_{\beta_i} \to X_{\zeta_i}, \;\;\;\;\;\;\; k_i \colon X_{\beta_{i-1}} \to X_{\zeta_i}, \;\;\;\;\;\;\; i \in \Z/N\Z,$$
the structure morphisms of $X$. 

Define $\L$ to be the full subcategory of $\Ho(\M^{\C_N})$ consisting of those diagrams $X \in \Ho(\M^{\C_N})$ which satisfy the following conditions ($\pi_*(-)$ denotes $[S,-]_*$):

\;

\;

\rm (i) The objects $X_{\beta_i}$ and $X_{\zeta_i}$ are cofibrant in $\M$ for any $i \in \Z/N\Z$;

\rm (ii) The graded $\pi_*S$-modules $\pi_*X_{\beta_i}$ and $\pi_*X_{\zeta_i}$ are objects of $\B[i]$ for any $i \in \Z/N\Z$;

\rm (iii) The map $\pi_*l_i \colon \pi_*X_{\beta_i} \to \pi_*X_{\zeta_i}$ is injective for any $i \in \Z/N\Z$.

\;

\;

Next, we recall the functor $\Q \colon \L \to \Mod  \pi_*S$. As a graded module it is given by 
$$\bigoplus_{i \in \Z/N\Z} \pi_*(\Cone(k_i)).$$ 
The differential is defined using the long exact sequences of mapping cones (see Section 3.3 of \cite{Pat12}).

Let $\K$ denote the full subcategory of $\L$ consisting of those objects $X$ which satisfy
$$\proj \pi_*X_{\beta_i} < N-1, \;\;\;\;\;\;\; \proj \pi_*X_{\zeta_i} < N-1$$
for all $i \in \Z/N\Z$. The following is proved in \cite[Section 3.3]{Pat12}:

\begin{prop} \label{Qequi} Let $\M$ be a simplicial stable model category with a compact generator $S$ of $\Ho(\M)$. Suppose that the graded ring $\pi_*S=[S,S]_*$ is $N$-sparse for $N \geq 2$ and the graded global homological dimension of $\pi_*S$ is less than $N$. Then the restricted functor 
$$\Q\vert_{ \K} \colon \K \to \Mod  \pi_*S$$ 
is fully faithful and a differential graded module $(C,d)$ is in the essential image of $\Q$ if and only if 
$$\proj \Ker d <N-1 \;\;\;\;\; \text{and} \;\;\;\;\; \proj \Imm d < N-1.$$ 
In particular, if $\gl \pi_*S$ is less than $N-1$, then $\Q \colon \L \to \Mod  \pi_*S$ is an equivalence of categories. 

\end{prop} 

Using this proposition we can once and for all choose an inverse $\Q^{-1} \colon \Q(\K) \to \K$. Denote by $\R'$ the composite
$$\xymatrix{ \Q(\K) \ar[rr]^-{(\Q\vert_{ \K})^{-1}} & & \K \subseteq \Ho(\M^{\C_N}) \ar[rr]^-{\Hocolim} & & \Ho(\M).}$$
By \cite[Section 3.4]{Pat12} one has a natural isomorphism $\pi_* \circ \R' \cong H_*$. Further note that by \cite[Remark 2.2.5]{H97} every cofibrant differential graded $\pi_*S$-module is projective as an underlying graded module and hence $\Q(\K)$ contains all cofibrant objects. Altogether we conclude that $\R'$ gives rise to a functor
$$\R \colon \D(\pi_*S) \to \Ho(\M)$$
with a natural isomorphism $\pi_* \circ \R \cong H_*$.

It is shown in \cite[Section 3.5]{Pat12} that the functor $\R$ commutes with suspensions. More precisely consider the functor $(-)^\# \colon \L \to \L$ defined by
$$X^\#_{\beta_i}= S^1 \wedge X_{\beta_{i-1}}, \;\;\;\;\;\;  X^\#_{\zeta_i}= S^1 \wedge X_{\zeta_{i-1},}$$
$$k_i^\#= 1 \wedge k_{i-1}, \;\;\;\;\;\;\; l_i^\#= 1 \wedge l_{i-1}, \;\;\;\;\;\; i \in \Z/N\Z.$$
Then there is natural isomorphism of differential graded $\pi_*S$-modules $\Q(X^\#) \cong \Q(X)[1]$. This induces a natural isomorphism $\Q^{-1} (M[1]) \cong \Q^{-1}(M)^\#$ and after applying the $\Hocolim$ we get a natural isomorphism $\xi' : \R(X[1]) \cong \Sigma \R(X)$. Let $\xi$ denote $-\xi'$. We fix for the rest of the paper the suspension isomorphism $\xi : \R(X[1]) \cong \Sigma \R(X)$. The reason for changing the sign becomes apparent when one tries to show that $\R$ sends certain distinguished triangles to distinguished triangles (as opposed to anti-distinguished triangles).

\section{Classification of module spectra in global dimension three} \label{mainsec}

\subsection{The main theorem}

We begin with the following important lemma:

\begin{lem} \label{presprojtr} Let $\M$ be a simplicial stable model category and $S$ a compact generator. Suppose that the graded ring $\pi_*S=[S,S]_*$ is $N$-sparse for $N \geq 2$ and the graded global homological dimension of $\pi_*S$ is less than $N$. Further let
$$\xymatrix{P \ar[r]^-f & Q \ar[r]^-i & C \ar[r]^-{\pa} & P[1] }$$
be a distinguished triangle in $\D(\pi_*S)$ such that $H_*P$ and $H_*Q$ are projective graded $\pi_*S$-modules. Then the triangle
$$\xymatrix{\R(P) \ar[r]^-{\R(f)} & \R(Q) \ar[r]^-{\R(i)} & \R(C) \ar[r]^-{\R(\pa)} & \R(P[1]) \cong \Sigma \R(P)}$$
is distinguished. 

\end{lem}

\begin{proof} Using \cite[Lemma 2.2.2]{Pat12} without loss of generality we may assume that $P$ and $Q$ have no differentials and $f \colon P \to Q$ is a morphism of graded $\pi_*S$-modules. Moreover, one can also assume that both $P$ and $Q$ are contained in $\B[n]$ for some $n$. We will now check that the functor $\R$ sends the algebraic mapping cone sequence 
$$\xymatrix{P \ar[r]^-f & Q \ar[r]^-i & C(f) \ar[r]^-{\pa} & P[1]}$$
to a distinguished triangle. 

By \cite[Proposition 3.2.1]{Pat12}, there exist fibrant-cofibrant objects $A$, $B$ and $V$ in $\M$ such that the following isomorphisms hold in $\Grmod \pi_*S$: 
$$\pi_* A \cong P, \;\;\;\;\;\; \pi_* B \cong Q, \;\;\;\;\;\; \pi_* V \cong \Imm f.$$
Moreover, we can choose cofibrations $\alpha \colon A \to V$ and $\gamma \colon V \to B$ (by changing $V$ and $B$ up to weak equivalence if necessary) so that under the latter isomorphisms $\pi_*\alpha$ corresponds to the map $f \colon P \twoheadrightarrow \Imm f$ and $\pi_*\gamma$ corresponds to the inclusion $\Imm f \hookrightarrow Q$. Consider the composite cofibration $\varphi= \gamma \circ \alpha \colon A \to B$. Then $\pi_*\varphi$ corresponds to $f \colon P \to Q$ under the isomorphisms $\pi_* A \cong P$ and $\pi_* B \cong Q$. Next, using the mapping cone sequence
$$\xymatrix{A \ar[r]^-{\alpha} & V \ar[r]^-{j} & \Cone (\alpha) \ar[r]^-{\delta} & S^1 \wedge A}$$
we see that $\pi_* \Cone (\alpha)$ is isomorphic to $\Ker f [1]$ in $\Grmod \pi_*S$. Moreover, we can choose an isomorphism $\pi_* \Cone (\alpha) \cong \Ker f [1]$ so that the map 
$$\pi_*\delta \colon \pi_*(\Cone (\alpha)) \to \pi_*(S^1 \wedge A) \cong \pi_*(A)[1]$$
corresponds to the inclusion $\Ker f[1] \hookrightarrow P[1]$. Further since $P$ and $Q$ are projective and $\gl \pi_*S < N$, it follows that the projective dimensions of $\Imm f$ and $\Ker f$ in $\Grmod \pi_*S$ are less than $N-1$. Hence the diagram
$$\xymatrix{Z \colon & \ast
   \ar@{<-}[drrrrrrr]%
    |<<<<<<<<<<<{\text{\Large \textcolor{white}{$\blacksquare$}}}%
    |<<<<<<<<<<<<<<<<<<<<<{\text{ \Large \textcolor{white}{$\blacksquare$}}}
    |<<<<<<<<<<<<<<<<<<<<<<<<<<<<<<<<<<<<<{\text{\Large \textcolor{white}{$\blacksquare$}}}
    |<<<<<<<<<<<<<<<<<<<<<<<<<<<<<<<<<<<<<<<<<<<<{\text{\Large \textcolor{white}{$\blacksquare$}}}
    |>>>>>>>>>>>>>>>>>>>>>>>>>>>>>>>>>>>>>>>>>>>>>>>{\text{\Large \textcolor{white}{$\blacksquare$}}}
    |>>>>>>>>>>>>>>>>>>>>>>>>>>>>>>>>>>>>>>>>{\text{\Large \textcolor{white}{$\blacksquare$}}}
    |>>>>>>>>>>>>>>>>>>>>>>>>>>>>>>{\text{\Large \textcolor{white}{$\blacksquare$}}}
    |>>>>>>>>>>>>>>>>>>>>>>>>>{\text{\Large \textcolor{white}{$\blacksquare$}}}
& \ast & \ldots &  \ast & B & \Cone(\alpha) & \ldots & \ast \\ & \ast \ar[u] \ar[ur] & \ast \ar[u] \ar[ur] & \ldots & \ast \ar[u] \ar[ur] & V \ar[u]^(.4){\gamma}  \ar[ur]^{j} & \ast \ar[u] \ar[ur] & \ldots& \ \ast, \ar[u]}$$
where $Z_{\zeta_n}=B$, $Z_{\beta_n}=V$ and $Z_{\zeta_{n+1}}=\Cone(\alpha)$ is contained in $\K$. 

Now let $X \in \K$ be the diagram with the only nontrivial entry $X_{\zeta_n}=A$ and let $Y \in \K$ be the diagram with the only nontrivial entry $Y_{\zeta_n}=B$. The map $\varphi= \gamma \circ \alpha \colon A \to B$ induces a map $\overline{\varphi}  \colon X \to Y$. Further, we have obvious maps $Y \to Z$ and $Z \to X^{\#}$ in $\K$, where the first map is the identity on $Y_{\zeta_n}=B$ and the second map is induced by $\delta \colon \Cone(\alpha) \to S^1 \wedge A$. By definition of the algebraic mapping cone, the sequence
$$\xymatrix{X \ar[r]^{\overline{\varphi}} & Y \ar[r] & Z \ar[r] & X^\#}$$ 
is mapped up to isomorphism by $\Q$ to the sequence 
$$\xymatrix{P \ar[r]^-{f} & Q \ar[r]^-{i} & C (-f) \ar[r]^-{\partial} & P[1].}$$
On the other hand the pushout diagrams
$$\xymatrix{A \ar[r]^{\alpha} \ar[d] & V \ar[d]^{j} \ar[r]^{\gamma} & B \ar[d] \\ CA \ar[r] & \Cone(\alpha) \ar[r] & \Cone(\varphi)}$$
show that the sequence 
$$\xymatrix{X \ar[r]^{\overline{\varphi}} & Y \ar[r] & Z \ar[r] & X^\#}$$
is mapped to the mapping cone sequence of $\varphi$ by the homotopy colimit functor. This finishes the proof. \end{proof}

The following is one of the main results of the paper. It improves the result obtained in Section 6.3 of \cite{Pat12}.

\begin{prop} \label{Rfulfaithful} Let $\M$ be a simplicial stable model category and $S$ a compact generator of $\Ho(\M)$. Suppose that the graded ring $\pi_*S=[S,S]_*$ is $N$-sparse for $N \geq 2$ and the graded global homological dimension of $\pi_*S$ is less than $N-1$. Further let $M$ be a differential graded $\pi_*S$-module with $\proj H_*M \leq 3$ in the category of graded $\pi_*S$-modules. Then for any differential graded $\pi_*S$-module $L$, the map
$$\R \colon [M,L] \to [\R(M), \R(L)]$$
is an isomorphism. \end{prop}

\begin{proof} By \cite[Lemma 2.2.2]{Pat12} there exists a morphism $\sigma \colon P \to M$ in $\D(\pi_*S)$, such that $H_*P$ is a projective graded $\pi_*S$-module and $H_*\sigma$ is an epimorphism. Embed this morphism into a distinguished triangle
$$\xymatrix{M' \ar[r]^{i} & P \ar[r]^{\sigma} & M \ar[r]^-{\pa} & M'[1].}$$
It follows from the long exact homology sequence that the map $H_*\pa \colon H_*(M) \to H_*(M'[1])$ is zero. Hence, we get a short exact sequence
$$\xymatrix{0 \ar[r] & H_*M' \ar[r]^-{H_*i} & H_*P \ar[r]^{H_*\sigma} & H_*M \ar[r] & 0.}$$
This implies that $\proj H_*M' \leq 2$. Now one can resolve $M'$. More precisely, once again using \cite[Lemma 2.2.2]{Pat12} we know that there exists a map $\lambda \colon Q \to M'$ in $\D(\pi_*S)$, such that $H_*Q$ is a projective graded $\pi_*S$-module and $H_*\lambda$ is an epimorphism. Next choose a distinguished triangle in $\D(\pi_*S)$
$$\xymatrix{M'' \ar[r]^j & Q \ar[r]^\lambda & M' \ar[r]^-{\delta} & M''[1].}$$
Since $H_*\lambda$ is an epimorphism, the long exact homology sequence implies that $H_*\delta=0$ and the sequence
$$\xymatrix{0 \ar[r] & H_*M'' \ar[r]^-{H_*j} & H_*Q \ar[r]^{H_*\lambda} & H_*M' \ar[r] & 0}$$
is exact. Hence we conclude that $\proj H_*M'' \leq 1$. 

Now using the octahedral axiom we combine the latter two distinguished triangles and get the following commutative diagram in $\D(\pi_*S)$
$$\xymatrix{ M[-1] \ar[rr]^{-\pa[-1]} \ar@{=}[d] & & M' \ar[r]^{i} \ar[d]_{\delta} & P \ar[r]^{\sigma} \ar@{-->}[d]^{\alpha} & M \ar@{=}[d]\\ M[-1] \ar[rr]^{-\delta \circ (\pa[-1])}  & & M''[1] \ar[d]_{j[1]} \ar[r]^{u}  & C \ar[r]^{v} \ar@{-->}[d]^{\beta} & M \ar[d]^{-\pa} \\ & & Q[1] \ar@{=}[r] \ar[d]_{\lambda[1]} & Q[1] \ar[d]^{(i \circ \lambda)[1]} \ar[r]^{\lambda[1]} & M'[1] \\ & & M'[1] \ar[r]^{i[1]} & P[1], & }$$ 
where the triangles in the rows and columns are distinguished. By Lemma \ref{presprojtr} the triangle
$$\xymatrix{\R(P) \ar[r]^{\R(\alpha)} & \R(C) \ar[r]^{\R(\beta)} & \R(Q[1]) \ar[rr]^-{\R((i \circ \lambda)[1])} & & \R(P[1]) \cong \Sigma \R(P)}$$
is distinguished. Consider the commutative diagram with exact rows
{\footnotesize
$$\xymatrix{[P[1],L] \ar[r] \ar[d]^\R & [Q[1], L] \ar[r] \ar[d]^\R & [C, L] \ar[r] \ar[d]^\R & [P, L] \ar[r] \ar[d]^\R & [Q, L] \ar[d]^\R \\  [\R(P[1]), \R(L)] \ar[r] & [\R(Q[1]), \R(L)] \ar[r] & [\R(C), \R(L)]  \ar[r] & [\R(P), \R(L)] \ar[r] & [\R(Q), \R(L)].}$$}
All the morphisms except the middle one in this diagram are isomorphisms by \cite[Section 5.2]{Pat12}. Now the Five lemma implies that so is the middle map 
$$\R \colon [C, L] \to [\R(C), \R(L)].$$
Next, since $H_*u$ is a monomorphism and $\proj H_*M'' \leq 1$, Lemma 6.2.4 of \cite{Pat12} applies and we see that the triangle
$$\xymatrix{\R(M''[1]) \ar[r]^{\R(u)} & \R(C) \ar[r]^{\R(v)} & \R(M) \ar[rr]^-{\R((\delta [1]) \circ \pa)} & & \R(M''[2]) \cong \Sigma \R(M''[1])}$$
is distinguished. Again by \cite[Section 5.2]{Pat12}, we know that the map $\R \colon [M'', L]_* \to [\R(M''), \R(L)]_*$ is an isomorphism since $\proj H_*M'' \leq 1$. Finally, another Five lemma argument together with the previous paragraph shows that the map
$$\R \colon [M, L] \to [\R(M), \R(L)]$$
is an isomorphism. 

\end{proof}

Next we prove that the functor $\R$ is essentially surjective on objects with homological dimension at most three. 

\begin{prop} \label{Ressurj} Let $\M$ be a simplicial stable model category and $S$ a compact generator of $\Ho(\M)$. Suppose that the graded ring $\pi_*S=[S,S]_*$ is $N$-sparse for $N \geq 2$ and the graded global homological dimension of $\pi_*S$ is less than $N-1$. Let $X$ be an object of $\Ho(\M)$ such that $\proj \pi_*X \leq 3$. Then there exists an object $M \in \D(\pi_*S)$ such that $\R(M)$ is isomorphic to $X$.

\end{prop}

\begin{proof} The argument is similar to the one in the previous proposition. By \cite[Lemma 2.2.2]{Pat12} we can choose a map $\sigma \colon F \to X$ such that $\pi_*F$ is a projective graded $\pi_*S$-module and $\pi_*\sigma$ is surjective. Embed $\sigma$ into a distinguished triangle in $\Ho(\M)$
$$\xymatrix{Y \ar[r]^i & F \ar[r]^{\sigma} & X \ar[r]^{\pa} & \Sigma Y.}$$
The resulting short exact sequence
$$\xymatrix{0 \ar[r] & \pi_*Y \ar[r]^{\pi_*i} & \pi_*F \ar[r]^{\pi_*\sigma} & \pi_* X \ar[r] & 0}$$
shows that $\proj \pi_* Y \leq 2$. Next, again \cite[Lemma 2.2.2]{Pat12} tells us that there exists a map $\lambda \colon G \to Y$ such that $\pi_*G$ is projective and $\pi_*\lambda$ an epimorphism. Choose a distinguished triangle in $\Ho(\M)$
$$\xymatrix{Z \ar[r]^{j} & G \ar[r]^{\lambda} & Y \ar[r]^-{\delta} & \Sigma Z.}$$
This triangle induces a short exact sequence
$$\xymatrix{0 \ar[r] & \pi_*Z \ar[r]^{\pi_*j} & \pi_*G \ar[r]^{\pi_*\lambda} & \pi_*Y \ar[r] & 0}$$
which shows that $\proj \pi_* Z \leq 1$. Combining the latter two distinguished triangles using the octahedral axiom we get a commutative diagram in $\Ho(\M)$
$$\xymatrix{ \Sigma^{-1} X \ar[rr]^{-\Sigma^{-1}\pa} \ar@{=}[d] & & Y \ar[r]^{i} \ar[d]_{\delta} & F \ar[r]^{\sigma} \ar@{-->}[d]^{\alpha} & X \ar@{=}[d]\\ \Sigma^{-1} X \ar[rr]^{-\delta \circ (\Sigma^{-1}\pa)}  & & \Sigma Z \ar[d]_{\Sigma j} \ar[r]^{u}  & C \ar[r]^{v} \ar@{-->}[d]^{\beta} & X \ar[d]^{-\pa} \\ & & \Sigma G \ar@{=}[r] \ar[d]_{\Sigma \lambda} & \Sigma G \ar[d]^{\Sigma (i \circ \lambda)} \ar[r]^{\Sigma \lambda} & \Sigma Y \\ & & \Sigma Y \ar[r]^{\Sigma i} & \Sigma F, & }$$
where the triangles in the rows and columns are distinguished. By \cite[Section 5.2]{Pat12}, there exists a morphism $\omega \colon P \to Q$ and a commutative diagram
$$\xymatrix{\R(P) \ar[r]^{\R(\omega)} \ar[d]_{\cong} & \R(Q) \ar[d]^{\cong} \\ G \ar[r]^{i \circ \lambda} & F}$$
with vertical arrows isomorphisms. Embed $\omega$ into a distinguished triangle in $\D(\pi_*S)$
$$\xymatrix{P \ar[r]^\omega & Q \ar[r]^{a} & D \ar[r]^{b} & P[1].}$$
Since $\pi_* \circ \R \cong H_*$, it follows that $H_*P$ and $H_*Q$ are projective graded $\pi_*S$-modules. Hence by Lemma \ref{presprojtr}, the triangle 
$$\xymatrix{\R(P) \ar[r]^{\R(\omega)} & \R(Q) \ar[r]^{\R(a)} & \R(D) \ar[r]^-{\R(b)} & \R(P[1]) \cong \Sigma \R(P)}$$
is distinguished. By the axioms of triangulated categories we see that this triangle is isomorphic to the distinguished triangle
$$\xymatrix{G \ar[r]^{i \circ \lambda} & F \ar[r]^{\alpha} & C \ar[r]^-{-\beta} & \Sigma G.}$$ 
Hence we get an isomorphism $\R(D) \cong C$. Further, since $\proj \pi_*Z \leq 1$, it follows from \cite[Section 5.2]{Pat12} that there exists a morphism $\theta \colon K \to D$ and a commutative diagram
$$\xymatrix{\R(K) \ar[r]^{\R(\theta)} \ar[d]_{\cong} & \R(D) \ar[d]^{\cong} \\ \Sigma Z \ar[r]^{\Sigma u} & C}$$
with vertical arrows isomorphisms. Consider the distinguished triangle in $\D(\pi_*S)$
$$\xymatrix{K \ar[r]^{\theta} & D \ar[r] & M \ar[r] & K[1].}$$
Using the isomorphism $\pi_* \circ \R \cong H_*$, we see that $\proj H_*K \leq 1$ and $H_*\theta$ is injective. Hence Lemma 6.2.4 of \cite{Pat12} applies to the latter distinguished triangle and we get a distinguished triangle
$$\xymatrix{\R(K) \ar[r]^-{\R(\theta)} & \R(D) \ar[r] & \R(M) \ar[r] & \R(K[1]) \cong \Sigma \R(K).}$$
Finally the triangulated category axioms imply that $\R(M) \cong X$. \end{proof}

Now Proposition \ref{Rfulfaithful} combined with Proposition \ref{Ressurj} implies Theorem \ref{main} if we take $\M=\Mod R$.

\subsection{Examples}

\begin{exmp} The truncated Brown-Peterson spectrum $BP\langle 2 \rangle$ for a prime $p \geq 5$ satisfies the assumptions of Corollary \ref{maincoro}. Indeed, there is a ring isomorphism
$$\pi_*BP\langle 2 \rangle \cong \Z_{(p)}[v_1, v_2], \;\;\;\;\;\;\;\;\;\; |v_1|=2(p-1),\;\;|v_2| = 2(p^2-1)$$
and $\gl \Z_{(p)}[v_1, v_2] =3$ and $2(p-1) > 5$. Thus, the functor
$$\R \colon \D(\Z_{(p)}[v_1, v_2]) \longrightarrow \Ho(\Mod BP \langle 2 \rangle)$$
is an equivalence of categories. Next, $\Omega^{\infty}BP\langle 2 \rangle$ is not a product of Eilenberg-MacLane spaces by \cite[A.1.7]{Pat12}. Therefore the model categories $\Mod BP \langle 2 \rangle$ and $\Mod \Z_{(p)}[v_1, v_2]$  can not be connected by a zig-zag of Quillen equivalences. In particular, $\R \colon \D(\Z_{(p)}[v_1, v_2]) \longrightarrow \Ho(\Mod BP \langle 2 \rangle)$ is not derived from a zig-zag of Quillen equivalences. \end{exmp}

\begin{exmp} Another ring spectrum to which Corollary \ref{maincoro} applies is the Johnson-Wilson spectrum $E(3)$ for a prime $p \geq 5$. This follows from the ring isomorphism
$$\pi_*E(3) \cong \Z_{(p)}[v_1, v_2, v_3, v_3^{-1}], \;\;\;\;\;\;\;\;\;\; |v_1|= 2(p-1), \;\; |v_2| = 2(p^2-1), \;\; |v_3| = 2(p^3-1),$$
since $\gl \Z_{(p)}[v_1, v_2, v_3, v_3^{-1}]=3$ and $2(p-1) > 5$. Further, as in the previous example, $\Omega^{\infty} E(3)$ is not a product of Eilenberg-MacLane spaces \cite[A.1.8]{Pat12}. Therefore, there is no zig-zag of Quillen equivalences between the model categories $\Mod E(3)$ and $\Mod \Z_{(p)}[v_1, v_2, v_3, v_3^{-1}]$. In particular, the equivalence $\R \colon \D(\Z_{(p)}[v_1, v_2, v_3, v_3^{-1}]) \longrightarrow \Ho(\Mod E(3))$ does not come from a zig-zag of Quillen equivalences.

\end{exmp}

\section{Triangulated equivalences} 

In this section we show that some of the equivalences constructed in \cite{Pat12} are triangulated. More precisely we show that if $R$ is a symmetric ring spectrum such that $\pi_*R$ is $N$-sparse for $N \geq 5$ and $\gl \pi_*R =2$, then the equivalence
\[\R \colon \D(\pi_*R) \longrightarrow \Ho(\Mod R)  \]
is triangulated. This follows using the methods of \cite{Pat16}. We also discuss generalizations of the results of \cite{Pat12, Pat16} and the previous section to abelian categories and twisted chain complexes. At the end as a special case we get the well-known classification of $E(1)$-local spectra at an odd prime \cite{B85, F96}.

\subsection{Triangulated equivalences for module spectra} Let $\M$ be a simplicial stable model category and $S$ a compact generator of $\Ho(\M)$. Suppose that the graded ring $\pi_*S=[S,S]_*$ is $N$-sparse for $N \geq 5$ and the graded global homological dimension of $\pi_*S$ is equal to two. Then using the methods of Franke, the paper \cite{Pat16} constructs functors
$$\R \colon \Ho((\Mod \pi_*S)^{\P}) \to \Ho(\M^{\P})$$ 
for any finite poset $\P$ with $\dimm \P \leq 1$. Here $\dimm \P$ stands for the length of a maximal chain in $\P$. If $\dimm \P \leq 1$, then the global homological dimension of  $(\Grmod \pi_*S)^{\P}$ is at most $3$ (see , \cite{Mit68}). Using this fact, the arguments from \cite[Section 3.1]{Pat16} together with the results of the previous section imply the following:

\begin{prop} \label{diaequiv} Let $\M$ be a simplicial stable model category and $S$ a compact generator of $\Ho(\M)$. Suppose that the graded ring $\pi_*S=[S,S]_*$ is $N$-sparse for $N \geq 5$ and the graded global homological dimension of $\pi_*S$ is equal to $2$. Further let $\P$ be a finite poset such that $\dimm \P \leq 1$. Then the functor
$$\R \colon \Ho((\Mod \pi_*S)^{\P}) \to \Ho(\M^{\P})$$ 
is an equivalence of categories.

\end{prop}

This proposition now implies that the equivalence $\R \colon \D(\pi_*S) \to \Ho(\M)$ is triangulated. Indeed, it follows from \cite[Propositions 3.2.1-2]{Pat16}, that the functors $\R$ are compatible with the derivator structures. Moreover, they are compatible with homotopy Kan extensions coming from maps of at most one dimensional posets (this follows from the proof of \cite[Proposition 4.2.1]{Pat16}). Since the mapping cone sequences can be described using restrictions and homotopy Kan extensions coming from maps of at most one dimensional posets (\cite[Section 4.1]{Pat16}), we get the following result:

\begin{theo} \label{triaequi} Let $\M$ be a simplicial stable model category and $S$ a compact generator of $\Ho(\M)$. Suppose that $\pi_*S=[S,S]_*$ is $N$-sparse for some $N \geq 5$ and $\gl \pi_*S=2$. Then the functor
$$\R \colon \D(\pi_*S)=\Ho(\Mod \pi_*S) \to \Ho(\M)$$
is a triangulated equivalence.

\end{theo}

\begin{proof} The proof is exactly the same as that of \cite[Theorem 4.2.3]{Pat16}. The homological condition only enters when concluding that the homotopy categories of diagrams with at most one dimensional posets are equivalent. Once this is ensured, the proof of \cite[Theorem 4.2.3]{Pat16} applies verbatim to the present case.

\end{proof}

\begin{exmp} The truncated Brown-Peterson spectrum $BP\langle 1 \rangle$ for a prime $p \geq 5$ satisfies the assumptions of Theorem \ref{maintr}. Thus, the equivalence
$$\R \colon \D(\Z_{(p)}[v_1]) \longrightarrow \Ho(\Mod BP \langle 1 \rangle)$$
is triangulated for $p \geq 5$. \end{exmp}

\begin{exmp} The Johnson-Wilson spectrum $E(2)$ for a prime $p \geq 5$ also satisfies the conditions of Theorem \ref{maintr}. We get that the equivalence 
$$\R \colon \D(\Z_{(p)}[v_1, v_2, v_2^{-1}]) \longrightarrow \D(E(2))$$ 
is triangulated for $p \geq 5$. 

\end{exmp}

\subsection{Franke's theorem at height one} 

In this subsection we will show how the results of \cite{Pat12, Pat16} and Section \ref{mainsec} can be generalized to abelian categories and twisted chain complexes. As a special case, one gets an algebraic classification of $E(1)$-local spectra at an odd prime. All the constructions here are due to Franke \cite{F96} (see also \cite{B85, R08}). Our results fill the gap which was discussed in the introduction. We skip here details of the proofs as they are identical to those in \cite{Pat12, Pat16} and Section \ref{mainsec}.

Let $\M$ be a simplicial stable model category and $\A$ an abelian category. Throughout this subsection we will assume that $\A$ has enough projectives and finite global homological dimension. We start with a setup for having an Adams spectral sequence in $\Ho(\M)$ based on $\A$. 

Suppose we are given an autoequivalence $[1] \colon \A \to \A$, and $F \colon \Ho(\M) \to \A$ a homological functor with a natural isomorphism $F(\Sigma X) \cong F(X)[1]$. Further assume that the following two conditions hold: 

{\rm (i)} If $C$ is an object of $\Ho(\M)$ and $F(C)$ is projective in $\A$, then the map
$$ F \colon [C,X] \longrightarrow \Hom_{\A}(F(C), F(X))$$
is an isomorphism for any object $X$ of $\Ho(\M)$.

{\rm (ii)} For any projective object $P$ in $\A$, there exists $G$ in $\Ho(\M)$ such that
$$F(G) \cong P$$
holds in $\A$.

\noindent Then there is a bounded convergent spectral sequence
$$E_2^{pq}=\Ext^p_{\A}(F(X), F(Y)[q]) \Rightarrow [X, \Sigma^{p+q}Y].$$ 
Here $[q] \colon \A \to \A$ denotes the iteration of $[1] \colon \A \to \A$ (or of its chosen inverse) (see e.g., \cite{C98}). 

It follows from the spectral sequence that under our conditions, the functor $F$ is conservative. 

Now we repeat the procedure of \cite[Section 3.3]{Pat12}. Suppose we are given a Serre subcategory $\B \subset \A$ which is closed under $[\pm N]$ for $N \geq 2$ and suppose that the functor
$$\bigoplus_{0 \leq i <N}\B \longrightarrow \A, \;\;\;\;\;\; (B_i)_{0 \leq i <N} \mapsto \bigoplus_{0 \leq i <N} B_i[i] $$
is an equivalence of categories. Let $C^{([1],1)}(\A)$ denote the category of \emph{twisted $([1],1)$-chain complexes}. An object of $C^{([1],1)}(\A)$ is just an object $C \in \A$ together with a morphism (the differential)
$$d \colon C \longrightarrow C[1]$$
such that $d[1] \circ d=0$. Note that this category is equivalent to the category of twisted $([N],N)$-chain complexes in $\B$ \cite{F96}. The category $C^{([1],1)}(\A)$ admits a stable model structure where weak equivalences are homology isomorphisms and fibrations are surjections (see e.g., \cite{F96}, \cite{BR11}). Let $\D^{([1],1)}(\A)$ denote the homotopy category of $C^{([1],1)}(\A)$. Then the homology functor
$$H \colon \D^{([1],1)}(\A) \longrightarrow \A$$
satisfies the conditions (i) and (ii) above. This gives an Adams spectral sequence for computing abelian groups of morphisms in $\D^{([1],1)}(\A)$. 

Finally, we are ready to recall the construction of Franke's functor $\R \colon \D^{([1],1)}(\A) \longrightarrow \Ho(\M)$. From now on we assume that all the above conditions are satisfied. Additionally assume that the global homological dimension of $\A$ is less than $N-1$. We will now again use the diagram category $\M^{\C_N}$. For an object $X \in \M^{\C_N}$, denote by 
$$l_i \colon X_{\beta_i} \to X_{\zeta_i}, \;\;\;\;\;\;\; k_i \colon X_{\beta_{i-1}} \to X_{\zeta_i}, \;\;\;\;\;\;\; i \in \Z/N\Z,$$
the structure morphisms of $X$. Consider the full subcategory $\L$ of $\Ho(\M^{\C_N})$ consisting of those diagrams $X \in \Ho(\M^{\C_N})$ which satisfy the following conditions:

\;

\;

\rm (i) The objects $X_{\beta_i}$ and $X_{\zeta_i}$ are cofibrant in $\M$ for any $i \in \Z/N\Z$;

\rm (ii) The objects $F(X_{\beta_i})$ and $F(X_{\zeta_i})$ are contained in $\B[i]$ for any $i \in \Z/N\Z$;

\rm (iii) The map $F(l_i) \colon F(X_{\beta_i}) \to F(X_{\zeta_i})$ is a monomorphism for any $i \in \Z/N\Z$.

\;

\;

As in \cite{Pat12} there is a functor $\Q \colon \L \longrightarrow C^{([1],1)}(\A)$, given by $\Q(X)= (\oplus_{i \in \Z/N\Z} F(\Cone(k_i)),d)$ which is an equivalence of categories. The proof of the fact that $\Q$ is an equivalence is a verbatim translation of the proof in Section 3.3 of \cite{Pat12} and uses the above spectral sequence. The functor $\R \colon \D^{([1],1)}(\A) \longrightarrow \Ho(\M)$ is then defined to be induced by the composite
$$\xymatrix{C^{([1],1)}(\A) \ar[r]^-{\Q^{-1}} & \L \ar[r]^-{\Hocolim} & \Ho(\M).}$$
It follows from the construction of $\R$ that it commutes with suspensions and that $F \circ \R \cong H$, where $H$ is the homology functor. 

The essential facts that fill the gap in Franke's paper for low dimensional cases are the following analogs of \cite[Lemmas 6.2.1 and 6.2.4]{Pat12} and Lemma \ref{presprojtr}. The proofs of the latter statements are axiomatic and also apply to the general context of abelian categories discussed here. 

\begin{lem} \label{lemgap1} Suppose that the global homological dimension of $\A$ is equal to two and $N \geq 4$. Let $B$ and $C$ be objects of $C^{([1],1)}(\A)$ and suppose $B$ has a zero differential. Further let
$$f \colon B \longrightarrow C$$
be a map in $C^{([1],1)}(\A)$ which induces a monomorphism on homology. Then the functor $\R$ sends the mapping cone sequence of $f$ to a distinguished triangle. 

\end{lem}

\begin{lem} \label{lemgap2} Suppose that the global homological dimension of $\A$ is equal to two and $N \geq 4$. Let $g \colon M \longrightarrow N$ be a map in the derived category $\D^{([1],1)}(\A)$ such that the homology $H(M)$ has projective dimension at most one and $H(g)$ is a monomorphism. Then $\R$ sends any distinguished triangle involving $g$ to a distinguished triangle.

\end{lem}

\begin{lem} \label{lemgap3} Distinguished triangles having at least two objects with projective homology are send to distinguished ones by the functor $\R$.

\end{lem}

Using these facts we arrive to the following statement (See \cite[Section 6.3]{Pat12} and Section \ref{mainsec} for the proof):

\begin{theo} \label{Franketheo1} Under the above assumptions the functor $\R \colon \D^{([1],1)}(\A) \longrightarrow \Ho(\M)$ is an equivalence of categories if the global homological dimension of $\A$ is equal to two and $N \geq 4$ or if the global homological dimension of $\A$ is equal to three and $N \geq 5$.
\end{theo}

It follows from the latter theorem that if the global homological dimension of $\A$ is equal to two and $N \geq 5$, then we also have an equivalence
$\R \colon \D^{([1],1)}(\A^{\P}) \longrightarrow \Ho(\M^{\P})$
for any at most one dimensional poset $\P$. These equivalences are compatible with the derivator structures. Now repeating the arguments of the proof of \cite[Theorem 4.2.3]{Pat16}, we obtain:

\begin{theo} \label{Franketheo2} Under the above assumptions the functor $\R \colon \D^{([1],1)}(\A) \longrightarrow \Ho(\M)$ is a triangulated equivalence if the global homological dimension of $\A$ is equal to two and $N \geq 5$.
\end{theo}

\begin{exmp} Let $R$ be a symmetric ring spectrum such that $\pi_* R$ is concentrated in degrees divisible by a natural number $N \geq 5$ and assume that the graded global homological dimension of $\pi_*R$ is equal to two. Then the model category $\Mod R$, the abelian category $\Grmod \pi_*R$ and the functor $\pi_* \colon \Ho(\Mod R) \longrightarrow \Grmod \pi_*R$ satisfy the conditions of Theorem \ref{Franketheo2}. 

\end{exmp}

\begin{exmp} Let $L_1\Sp$ denote the model category of $E(1)$-local spectra at an odd prime $p$. One has the homology theory
$$E(1)_* \colon \Ho(L_1\Sp) \to E(1)_*E(1)-\Comod.$$
The category $E(1)_*E(1)-\Comod$ is an abelian category with enough injectives and its global cohomological dimension is equal to two \cite{B85}. By passing to the dual categories, we see that the conditions of Theorem \ref{Franketheo1} are satisfied (here $N=2(p-1)$). Consequently one gets that the functor
$$\R \colon \D^{([-1],1)}((E(1)_*E(1)-\Comod)^{op}) \longrightarrow \Ho(L_1\Sp)^{op}$$
is an equivalence of categories. Moreover if $p \geq 5$, by Theorem \ref{Franketheo2}, the functor $\R$ is a triangulated equivalence. By dualizing back one obtains that 
$$\R^{op} \colon \D^{([1],1)}(E(1)_*E(1)-\Comod) \longrightarrow \Ho(L_1\Sp)$$
is an equivalence of categories and for $p \geq 5$ a triangulated equivalence. 
%Here an object of $\D^{([-1],1)}(E(1)_*E(1)-\Comod)$ is an $E(1)_*E(1)$-comodule $C$ together with a differential 
%
%$$d : C \to C[-1].$$ 
%
The functor $\R^{op}$ coincides with the functor constructed by Franke. The $op$ shows up here because Franke uses injective objects and coimages to construct his functor, whereas we used projective objects and kernels. The question, whether the equivalence is triangulated in the case $p=3$, remains open.
\end{exmp}

\vspace{0.5cm}

\noindent Mathematisches Institut \\ Universit\"at Bonn \\Endenischer Allee 60 \\ 53115 Bonn \\ Germany

\;

\;

\noindent E-mail address: \texttt{irpatchk@math.uni-bonn.de}

\end{document}